\documentclass[10pt]{article}
\usepackage[latin1]{inputenc}
\usepackage[english]{babel}
\usepackage[T1]{fontenc}
\usepackage{amsmath}
\usepackage{amsthm}
\usepackage{amssymb}
\usepackage{latexsym}
\usepackage{graphicx}
\usepackage{tikz}
\usepackage{lmodern}
\usepackage{scalefnt}
\usepackage{MnSymbol}
\usetikzlibrary{positioning}
\usetikzlibrary{shapes,arrows}
\usepackage{subfigure}
\usepackage{pgfplots}
\usepackage{caption}
\usepackage{cases}
\usepackage{hyperref}
\pgfplotsset{compat=1.7}
\newtheorem{definition}{Definition}[section]
\newtheorem{theorem}{Theorem}[section]
\newtheorem{example}{Example}[section]

\newtheorem{conjecture}{Conjecture}[section]
\newtheorem{proposition}{Proposition}[section]

\begin{document}

\title{Some further applications of a lattice theoretic method in the study of singular LCM matrices}

\author{{\sc Mika Mattila and Pentti Haukkanen}
\\Faculty of Information Technology and Communication Sciences,\\Tampere University, Finland
\\E-mail:  mika.mattila@tuni.fi, pentti.haukkanen@tuni.fi\\ \\
{\sc Jori Mäntysalo}
\\IT services,\\Tampere University, Finland
\\E-mail: jori.mantysalo@tuni.fi}
\maketitle

\begin{abstract}
In 1876 H. J. S. Smith defined an LCM matrix as follows: let $S=\{x_1,x_2,\ldots,x_n\}$ be a set of positive integers with $x_1<x_2<\cdots<x_n$. The LCM matrix $[S]$ on the set $S$ is the $n\times n$ matrix with $\mathrm{lcm}(x_i,x_j)$ as its $ij$ entry.
During the last 30 years singularity of LCM matrices has interested many authors. In 1992 Bourque and Ligh ended up conjecturing that if the GCD closedness of the set $S$ (which means that $\gcd(x_i,x_j)\in S$ for all $i,j\in\{1,2,\ldots,n\}$), suffices to guarantee the invertibility of the matrix $[S]$. However, a few years later this conjecture was proven false first by Haukkanen et al. and then by Hong. It turned out that the conjecture holds only on GCD closed sets with at most 7 elements but not in general for larger sets. However, the given counterexamples did not give much insight on why does the conjecture fail exactly in the case when $n=8$. This situation was later improved in a couple of articles, where a new lattice theoretic approach was introduced (the method is based on the fact that because the set $S$ is assumed to be GCD closed, the structure $(S,|)$ actually forms a meet semilattice). For example, it has been shown that in the case when the set $S$ has $8$ elements and the matrix $[S]$ is singular, there is only one option for the semilattice structure of $(S,|)$, namely the cube structure.

Since the cases $n\leq 8$ have been thoroughly studied in various articles, the next natural step is to apply the methods to the case $n=9$. This was done by Alt{\i}n{\i}\c{s}{\i}k and Alt{\i}nta\c{s} as they consider the different lattice structures of $(S,|)$ with nine elements that can result in a singular LCM matrix $[S]$. However, their investigation leaves two open questions, and the main purpose of this presentation is to provide solutions to them. We shall also give a new lattice theoretic proof for a result referred to as Sun's conjecture, which was originally proven by Hong via number theoretic approach.
\end{abstract}

\section{Introduction}

The concept of an LCM matrix, as well as the concept of a GCD matrix, was originally defined by H. J. S. Smith \cite{Smi}
in his seminal paper from the year 1876. Assuming the set $S =\{x_1, x_2,\ldots, x_n\}$ to be a finite subset of $\mathbb{Z}^+$ with distinct elements, Smith defined the GCD matrix $(S)$ of the set $S$ to be the $n\times n$ matrix with $\gcd(x_i, x_j)$ as its $ij$ element. Similarly, the LCM matrix $[S]$ of the set $S$ was defined to be the $n\times n$ matrix with $\mathrm{lcm}(x_i, x_j)$ as its $ij$ element. Besides calculating determinant formulas for several GCD and LCM type matrices, Smith was also considered the invertibility of GCD and LCM matrices. For example, he showed that if the set $S$ is factor closed (i.e. the implication ($y\,|\,x$ for some $x\in S)\Rightarrow y\in S$ holds for all $y\in\mathbb{Z}^+$), then both of the matrices $(S)$ and $[S]$ are invertible.

Surprisingly it took more than 110 years until LCM matrices got as get much attention in mathematical literature. Although GCD-type matrices had been studied in several papers over the years, it was not until 1992 that LCM matrices were reintroduced by Bourque and Ligh \cite{Bour92}. Among other things, they showed that it is actually quite easy to find singular LCM matrices by considering the LCM matrix of the four element set $S = \{1, 2, 15, 42\}$ (see \cite[p. 68]{Bour92}). The authors then turned the attention to finding conditions less restricting than factor closedness of the set $S$ that would suffice to quarantee the invertibility of the matrix $[S]$. They ended up conjecturing that the GCD-closedness of the set $S$ suffices to guarantee the invertibility of $[S]$.

In 1997 Haukkanen et al. \cite{HauWanSil} managed to find a GCD closed set $S$ with $9$ elements whose LCM matrix is singular, which disproved the Bourque-Ligh conjecture and showed that GCD closedness of the set $S$ does not actually suffice to guarantee the invertibility of the matrix $[S]$. Two years later Hong \cite{Hong99} was able to find another similar counterexample in which there was only $8$ elements in the set $S$. By using number theoretic methods Hong \cite{Hong06} also showed that the conjecture holds for GCD closed sets with at most 7 elements and that it does not hold in general for larger sets, which in some sense meant that the conjecture was solved completely. However, if we assume the set $S$ is GCD closed, it also means that the structure $(S,|)$ itself forms a meet semilattice, which makes it possible to study the conjecture from an entirely lattice theoretic point of view. In \cite{KMH18} Korkee et al. showed (via investigation of all possible semilattice structures with at most $7$ elements) that the LCM matrix $[S]$ is invertible for any GCD closed set $S$ with $|S|\leq 7$. In \cite{MHM15} this same lattice theoretic approach is utilized to show that if the matrix $[S]$ is singular and the set $S$ is GCD closed with 8 elements, then $(S,|)$ has unique, cube-like structure. These same methods were also adapted by Alt{\i}n{\i}\c{s}{\i}k et al. in \cite{AA17},
where they study the singularity of the matrix $[S]$ in the case when $S$ is a GCD closed set with $9$ elements. Finally, in 2020 Haukkanen et. al. \cite{MHM20} refined the lattice theoretic method by defining a property that was in common for all meet semilattices with at most $8$ elements that sufficed to guarantee the invertibility of the corresponding LCM matrix $[S]$. At the same time it also turned out that in many cases the lattice theoretic structure of $(S,|)$ not only guarantees the invertibility but also completely determines the inertia of the matrix $[S]$.

In the present article we continue utilizing our lattice theoretic method and find a couple of new ways to apply it. In Section 3 we reconsider the work done by Alt{\i}n{\i}\c{s}{\i}k and Alt{\i}nta\c{s} in \cite{AA17} from a slightly different perspective and we classify all possible $9$ element semilattice structures for which there exists a gcd closed set of this type such that the corresponding LCM matrix is singular. In this context we also solve one open conjecture about this 9 element case raised in their paper. In Section 4 we turn our attention to another conjecture raised by Alt{\i}n{\i}\c{s}{\i}k and Alt{\i}nta\c{s} and disprove it by presenting a couple of counterexamples. In Section 5 we give a novel lattice theoretic proof for a result known as Sun's conjecture, and finally in Section 6 we give an example on how our method can give some useful information in the study of the so-called power LCM matrices as well.

\section{Preliminaries}

First we need to introduce a couple of notations and concepts developed in \cite{MHM20} and bring them to our lattice theoretic context. Let $S=\{x_1,x_2,\ldots,x_n\}$ be a GCD closed set with $x_i\preceq x_j\Rightarrow i\leq j$. Let us denote
\[
C_S(x)=\{y\in S\quad \big|\quad y\,|\, x \text{\ and\ for\ all\ }z\in S: (y\,|\,z\text{\ and\ }z\,|\, x\Rightarrow y=z)\}
\] 
and
\[
\mathrm{meetcl}_S(C_S(x))=\{\gcd(y_1,y_2,\ldots,y_k)\ \big|\ k\in\mathbb{Z}^+\text{\ and\ }y_1,y_2,\ldots,y_k\in C_S(x)\}.
\]
In other words, $C_S(x)$ is the set of all elements of the set $S$ that are covered by $x$ in the meet semilattice $(S,|)$ and $\mathrm{meetcl}_S(C_S(x))$ is the smallest GCD closed subset of $S$ that contains all the elements of the set $C_S(x)$.
	
\begin{definition}[cf. \cite{MHM20}, Definition 2.2]
An element $x\in S$ generates a double-chain set in $S$ if the set $\mathrm{meetcl}(C_S(x))\setminus C_S(x)$ can be expressed as a union of two disjoint sets $A$ and $B$ that are chains in $(S,|)$.
\end{definition}

By the simple fact that $\gcd(x_i,x_j)\text{lcm}(x_i,x_j)=x_ix_j$ we may decompose the matrix $[S]$ as
\[
[S]=\text{diag}(x_1,x_2,\ldots,x_n)\left(\frac{1}{\gcd(x_i,x_j)}\right)\text{diag}(x_1,x_2,\ldots,x_n).
\]
Next we use Möbius inversion and define the function $\Psi_{S}$ on $S$ as
\begin{equation}\label{eq:psi}
\Psi_{S}(x_i)=\sum_{x_j\,|\,x_i}\frac{\mu_S(x_j,x_i)}{x_j},
\end{equation}
where the Möbius function values $\mu_S(x_j,x_i)$ are defined by using the recursive formula
\begin{align*}
&\mu_S(x_i,x_i)=1,\\
&\mu_S(x_j,x_i)=-\sum_{x_j\prec x_k\preceq x_i}\mu_S(x_k,x_i)=-\sum_{x_j\preceq x_k\prec x_i}\mu_S(x_j,x_k).
\end{align*}

Now the matrix $\left(\frac{1}{\gcd(x_i,x_j)}\right)$ may be written as
\[
\left(\frac{1}{\gcd(x_i,x_j)}\right)=E\,\text{diag}(\Psi_{S}(x_1),\Psi_{S}(x_2),\ldots,\Psi_{S}(x_n))\,E^T,
\]
where $E=(e_{ij})$ is the $0,1$ incidence matrix of the set $S$ with
\[
e_{ij}=
\begin{cases}
1 &\text{if }x_j\,|\,x_i,\\
0 &\text{otherwise.}
\end{cases}
\]

Putting all together we have
\begin{equation*}
[S]=\Delta E\Lambda E^T\Delta=(\Delta E)\Lambda(\Delta E)^T,
\end{equation*}
where $\Delta=\text{diag}(x_1,x_2,\ldots,x_n)$ and $\Lambda=\text{diag}(\Psi_{S}(x_1),\Psi_{S}(x_2),\ldots\Psi_{S}(x_n)).$

Since the matrix $\Delta E$ is clearly invertible (triangular matrix with nonzero diagonal elements), the matrix $[S]$ is invertible if and only if the matrix $\Lambda$ is invertible. Moreover, 
\[
\det\Lambda=\Psi_{S}(x_1)\Psi_{S}(x_2)\cdots\Psi_{S}(x_n).
\]
From this we easily obtain the following fundamental result.

\begin{proposition}\label{th:invertibility}
If the set $S$ is GCD closed, then the LCM matrix $[S]$ is invertible if and only if $\Psi_{S}(x_i)\neq0$ for all $i=1,2,\ldots,n.$ 
\end{proposition}

The following theorem is one of the key results of \cite{MHM20} and it plays a pivotal role in this article as well.

\begin{theorem}[cf. \cite{MHM20}, Theorem 4.1]\label{th:AB-joukko}
Let $S$ be a GCD closed set. If the element $x_i\in S$ generates a double-chain set in $S$, then $\Psi_{S}(x_i)\neq 0$.
\end{theorem}

The following result is important to keep in mind when one tries to construct a set $S$ whose LCM matrix $[S]$ is singular.

\begin{theorem}
Suppose that the element $x_i\in S$ and that there is at least one element in $C_S(x_i)$. If $x_i>\mathrm{lcm}(C_S(x_i))$, then $\Psi_{S}(x_i)\neq 0.$
\end{theorem}

\begin{proof}
Suppose that $x_i=a(\text{lcm}(C_S(x_i)))$, where $a\in\mathbb{Z}$ and $a>1$. Then
\begin{align*}
&\Psi_{S}(x_i)=\sum_{x_j\preceq x_i}\frac{\mu_S(x_j,x_i)}{x_j}\\
&=\frac{1}{\text{lcm}(C_S(x_i))}\left(\frac{\text{lcm}(C_S(x_i))}{x_i}+\sum_{x_j\prec x_i}\frac{\text{lcm}(C_S(x_i))}{x_j}\mu_S(x_j,x_i)\right)\\
&=\frac{1}{\text{lcm}(C_S(x_i))}\Bigg(\underbrace{\frac{1}{a}}_{\not\in \mathbb{Z}}+\sum_{x_j\prec x_i}\underbrace{\frac{\text{lcm}(C_S(x_i))}{x_j}}_{\in\mathbb{Z}}\underbrace{\mu_S(x_j,x_i)}_{\in\mathbb{Z}}\Bigg)\neq0.
\end{align*}
\end{proof}

\section{Solving the 9 element case completely}

The method used by Alt{\i}n{\i}\c{s}{\i}k and Alt{\i}nta\c{s} in \cite{AA17} is based on constructing a GCD closed set $S_8=\{x_1,x_2,\ldots,x_8\}$ for which the matrix $[S_8]$ is singular. They then go through all the different ways to insert a new positive integer $a$ in the set so that the set $S_9:=S_8\cup\{a\}$ would still remain GCD closed. If this turns out to be possible, the LCM matrix $[S_9]$ is also guaranteed to remain singular as it is similar to a matrix that has one of its leading principal minors equal to zero. Unfortunately, however, this method does not find the singular LCM matrices of order 9 whose every proper principal submatrix is nonsingular. This means that if we are interested in finding all possible semilattice structures with $9$ elements that correspond to some singular LCM matrix $[S]$, the method of Alt{\i}n{\i}\c{s}{\i}k and Alt{\i}nta\c{s} only gives a partial solution and cannot find all of them.

There are 5994 different meet semilattice structures with 9 elements (in comparison to only 1078 meet semilattice structures with 8 elements and 37622 structures with 10 elements). Since we are only interested in those structures that can be used to produce singular LCM matrices, Theorem \ref{th:AB-joukko} and Proposition \ref{th:invertibility} allow us to sieve out all the semilattice structures where every element generates a double-chain set. We are going to utilize program Sage in order to achieve this. Very slow but working way to generate all semilattices is
\\ \\
\verb+SLall = [MeetSemilattice(P) for P in Posets(9) if P.is_meet_semilattice()]+
\\ \\
And then we select only those possessing an element $e$ such that it covers at least three elements, and the subposet of the meet semilattice generated by those elements with generating elements removed has width greater than or equal to three:
\\ \\
\verb+SLspecial = []+\\
\verb+for SL in SLall:+\\
\verb+     for e in SL:+\\
\verb+         if len(SL.lower_covers(e)) >= 3:+\\
\verb+             elms = SL.lower_covers(e)+\\
\verb+             SL1 = SL.submeetsemilattice(elms)+\\
\verb+             P = SL1.subposet([e for e in SL1 if e not in elms])+\\
\verb+             if P.width() >= 3:+\\
\verb+                 Lspecial.append(L)+\\
\verb+                 break+\\

This leaves us with 13 semilattice structures shown in Figure \ref{fig:tapaukset} that we need to consider separately.

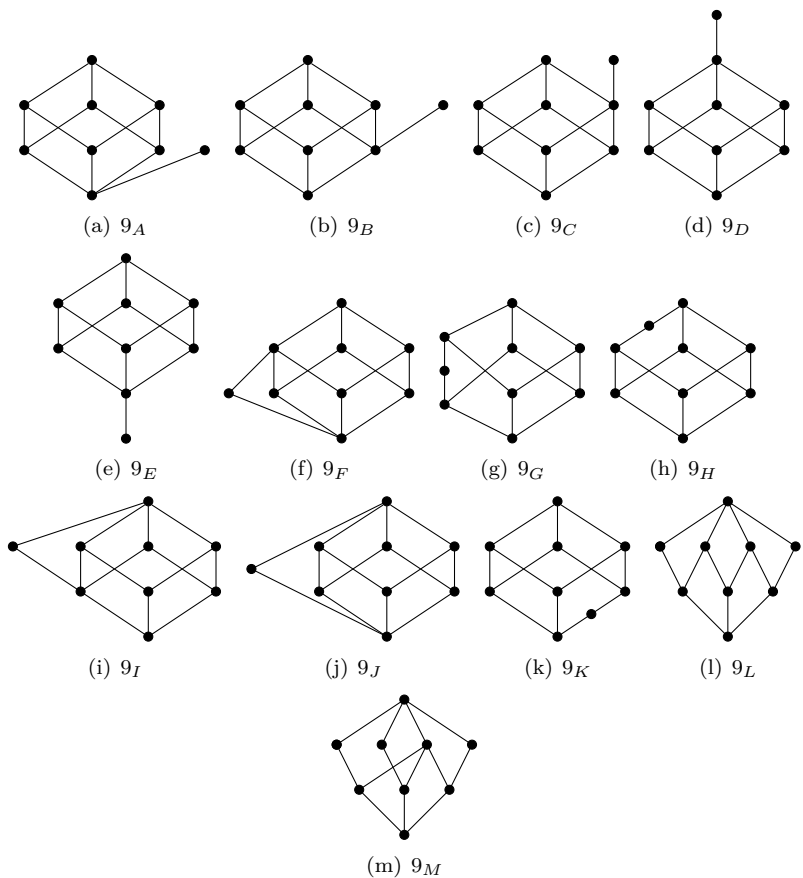
\begin{figure}[h!]
\centering
\subfigure[$9_A$]
{{\scalefont{0.8}
\begin{tikzpicture}[scale=0.6]
\draw (1,0)--(-0.5,1)--(-0.5,2)--(0.92,2.92);
\draw (1,0)--(1,1)--(-0.5,2);
\draw (1,0)--(3.5,1);
\draw (-0.5,1)--(1,2)--(1,2.9);
\draw (1,0)--(2.5,1)--(2.5,2)--(1.08,2.92);
\draw (1,1)--(2.5,2);
\draw (2.5,1)--(1,2);
\draw [fill] (-0.5,1) circle [radius=0.1];
\draw [fill] (1,1) circle [radius=0.1];
\draw [fill] (2.5,1) circle [radius=0.1];
\draw [fill] (1,0) circle [radius=0.1];
\draw [fill] (-0.5,2) circle [radius=0.1];
\draw [fill] (1,2) circle [radius=0.1];
\draw [fill] (2.5,2) circle [radius=0.1];
\draw [fill] (1,3) circle [radius=0.1];
\draw [fill] (3.5,1) circle [radius=0.1];
\end{tikzpicture}}
}
\subfigure[$9_B$]
{{\scalefont{0.8}
\begin{tikzpicture}[scale=0.6]
\draw (1,0)--(-0.5,1)--(-0.5,2)--(0.92,2.92);
\draw (1,0)--(1,1)--(-0.5,2);
\draw (2.5,1)--(4,2);
\draw (-0.5,1)--(1,2)--(1,2.9);
\draw (1,0)--(2.5,1)--(2.5,2)--(1.08,2.92);
\draw (1,1)--(2.5,2);
\draw (2.5,1)--(1,2);
\draw [fill] (-0.5,1) circle [radius=0.1];
\draw [fill] (1,1) circle [radius=0.1];
\draw [fill] (2.5,1) circle [radius=0.1];
\draw [fill] (1,0) circle [radius=0.1];
\draw [fill] (-0.5,2) circle [radius=0.1];
\draw [fill] (1,2) circle [radius=0.1];
\draw [fill] (2.5,2) circle [radius=0.1];
\draw [fill] (1,3) circle [radius=0.1];
\draw [fill] (4,2) circle [radius=0.1];
\end{tikzpicture}}
}
\subfigure[$9_C$]
{{\scalefont{0.8}
\begin{tikzpicture}[scale=0.6]
\draw (1,0)--(-0.5,1)--(-0.5,2)--(0.92,2.92);
\draw (1,0)--(1,1)--(-0.5,2);
\draw (-0.5,1)--(1,2)--(1,2.9);
\draw (1,0)--(2.5,1)--(2.5,2)--(1.08,2.92);
\draw (1,1)--(2.5,2);
\draw (2.5,1)--(1,2);
\draw (2.5,2)--(2.5,3);
\draw [fill] (-0.5,1) circle [radius=0.1];
\draw [fill] (1,1) circle [radius=0.1];
\draw [fill] (2.5,1) circle [radius=0.1];
\draw [fill] (1,0) circle [radius=0.1];
\draw [fill] (-0.5,2) circle [radius=0.1];
\draw [fill] (1,2) circle [radius=0.1];
\draw [fill] (2.5,2) circle [radius=0.1];
\draw [fill] (1,3) circle [radius=0.1];
\draw [fill] (2.5,3) circle [radius=0.1];
\end{tikzpicture}}
}
\subfigure[$9_D$]
{{\scalefont{0.8}
\begin{tikzpicture}[scale=0.6]
\draw (1,0)--(-0.5,1)--(-0.5,2)--(0.92,2.92);
\draw (1,0)--(1,1)--(-0.5,2);
\draw (-0.5,1)--(1,2)--(1,2.9);
\draw (1,0)--(2.5,1)--(2.5,2)--(1.08,2.92);
\draw (1,1)--(2.5,2);
\draw (2.5,1)--(1,2);
\draw (1,3)--(1,4);
\draw [fill] (-0.5,1) circle [radius=0.1];
\draw [fill] (1,1) circle [radius=0.1];
\draw [fill] (2.5,1) circle [radius=0.1];
\draw [fill] (1,0) circle [radius=0.1];
\draw [fill] (-0.5,2) circle [radius=0.1];
\draw [fill] (1,2) circle [radius=0.1];
\draw [fill] (2.5,2) circle [radius=0.1];
\draw [fill] (1,3) circle [radius=0.1];
\draw [fill] (1,4) circle [radius=0.1];
\end{tikzpicture}}
}
\subfigure[$9_E$]
{{\scalefont{0.8}
\begin{tikzpicture}[scale=0.6]
\draw (1,0)--(-0.5,1)--(-0.5,2)--(0.92,2.92);
\draw (1,0)--(1,1)--(-0.5,2);
\draw (-0.5,1)--(1,2)--(1,2.9);
\draw (1,0)--(2.5,1)--(2.5,2)--(1.08,2.92);
\draw (1,1)--(2.5,2);
\draw (2.5,1)--(1,2);
\draw (1,0)--(1,-1);
\draw [fill] (-0.5,1) circle [radius=0.1];
\draw [fill] (1,1) circle [radius=0.1];
\draw [fill] (2.5,1) circle [radius=0.1];
\draw [fill] (1,0) circle [radius=0.1];
\draw [fill] (-0.5,2) circle [radius=0.1];
\draw [fill] (1,2) circle [radius=0.1];
\draw [fill] (2.5,2) circle [radius=0.1];
\draw [fill] (1,3) circle [radius=0.1];
\draw [fill] (1,-1) circle [radius=0.1];
\end{tikzpicture}}
}
\subfigure[$9_F$]
{{\scalefont{0.8}
\begin{tikzpicture}[scale=0.6]
\draw (1,0)--(-0.5,1)--(-0.5,2)--(0.92,2.92);
\draw (1,0)--(1,1)--(-0.5,2);
\draw (-0.5,1)--(1,2)--(1,2.9);
\draw (1,0)--(2.5,1)--(2.5,2)--(1.08,2.92);
\draw (1,1)--(2.5,2);
\draw (2.5,1)--(1,2);
\draw (1,0)--(-1.5,1)--(-0.5,2);
\draw [fill] (-0.5,1) circle [radius=0.1];
\draw [fill] (1,1) circle [radius=0.1];
\draw [fill] (2.5,1) circle [radius=0.1];
\draw [fill] (1,0) circle [radius=0.1];
\draw [fill] (-0.5,2) circle [radius=0.1];
\draw [fill] (1,2) circle [radius=0.1];
\draw [fill] (2.5,2) circle [radius=0.1];
\draw [fill] (1,3) circle [radius=0.1];
\draw [fill] (-1.5,1) circle [radius=0.1];
\end{tikzpicture}}
}
\subfigure[$9_G$]
{{\scalefont{0.8}
\begin{tikzpicture}[scale=0.6]
\draw (1,0)--(-0.5,0.75)--(-0.5,2.25)--(1,3);
\draw (1,0)--(1,1)--(-0.5,2.25);
\draw (-0.5,0.75)--(1,2)--(1,3);
\draw (1,0)--(2.5,1)--(2.5,2)--(1,3);
\draw (1,1)--(2.5,2);
\draw (2.5,1)--(1,2);
\draw [fill] (-0.5,0.75) circle [radius=0.1];
\draw [fill] (1,1) circle [radius=0.1];
\draw [fill] (2.5,1) circle [radius=0.1];
\draw [fill] (1,0) circle [radius=0.1];
\draw [fill] (-0.5,2.25) circle [radius=0.1];
\draw [fill] (1,2) circle [radius=0.1];
\draw [fill] (2.5,2) circle [radius=0.1];
\draw [fill] (1,3) circle [radius=0.1];
\draw [fill] (-0.5,1.5) circle [radius=0.1];
\end{tikzpicture}}
}
\subfigure[$9_H$]
{{\scalefont{0.8}
\begin{tikzpicture}[scale=0.6]
\draw (1,0)--(-0.5,1)--(-0.5,2)--(0.92,2.92);
\draw (1,0)--(1,1)--(-0.5,2);
\draw (-0.5,1)--(1,2)--(1,2.9);
\draw (1,0)--(2.5,1)--(2.5,2)--(1.08,2.92);
\draw (1,1)--(2.5,2);
\draw (2.5,1)--(1,2);
\draw [fill] (-0.5,1) circle [radius=0.1];
\draw [fill] (1,1) circle [radius=0.1];
\draw [fill] (2.5,1) circle [radius=0.1];
\draw [fill] (1,0) circle [radius=0.1];
\draw [fill] (-0.5,2) circle [radius=0.1];
\draw [fill] (1,2) circle [radius=0.1];
\draw [fill] (2.5,2) circle [radius=0.1];
\draw [fill] (1,3) circle [radius=0.1];
\draw [fill] (0.25,2.5) circle [radius=0.1];
\end{tikzpicture}}
}
\subfigure[$9_I$]
{{\scalefont{0.8}
\begin{tikzpicture}[scale=0.6]
\draw (1,0)--(-0.5,1)--(-0.5,2)--(0.92,2.92);
\draw (1,0)--(1,1)--(-0.5,2);
\draw (-0.5,1)--(1,2)--(1,2.9);
\draw (1,0)--(2.5,1)--(2.5,2)--(1.08,2.92);
\draw (1,1)--(2.5,2);
\draw (-0.5,1)--(-2,2)--(1,3);
\draw (2.5,1)--(1,2);
\draw [fill] (-0.5,1) circle [radius=0.1];
\draw [fill] (1,1) circle [radius=0.1];
\draw [fill] (2.5,1) circle [radius=0.1];
\draw [fill] (1,0) circle [radius=0.1];
\draw [fill] (-0.5,2) circle [radius=0.1];
\draw [fill] (1,2) circle [radius=0.1];
\draw [fill] (2.5,2) circle [radius=0.1];
\draw [fill] (-2,2) circle [radius=0.1];
\draw [fill] (1,3) circle [radius=0.1];
\end{tikzpicture}}
}
\subfigure[$9_J$]
{{\scalefont{0.8}
\begin{tikzpicture}[scale=0.6]
\draw (1,0)--(-0.5,1)--(-0.5,2)--(0.92,2.92);
\draw (1,0)--(1,1)--(-0.5,2);
\draw (-0.5,1)--(1,2)--(1,2.9);
\draw (1,0)--(2.5,1)--(2.5,2)--(1.08,2.92);
\draw (1,1)--(2.5,2);
\draw (2.5,1)--(1,2);
\draw (1,0)--(-2,1.5)--(1,3);
\draw [fill] (-0.5,1) circle [radius=0.1];
\draw [fill] (1,1) circle [radius=0.1];
\draw [fill] (2.5,1) circle [radius=0.1];
\draw [fill] (1,0) circle [radius=0.1];
\draw [fill] (-0.5,2) circle [radius=0.1];
\draw [fill] (1,2) circle [radius=0.1];
\draw [fill] (2.5,2) circle [radius=0.1];
\draw [fill] (1,3) circle [radius=0.1];
\draw [fill] (-2,1.5) circle [radius=0.1];
\end{tikzpicture}}
}
\subfigure[$9_K$]
{{\scalefont{0.8}
\begin{tikzpicture}[scale=0.6]
\draw (1,0)--(-0.5,1)--(-0.5,2)--(0.92,2.92);
\draw (1,0)--(1,1)--(-0.5,2);
\draw (-0.5,1)--(1,2)--(1,2.9);
\draw (1,0)--(2.5,1)--(2.5,2)--(1.08,2.92);
\draw (1,1)--(2.5,2);
\draw (2.5,1)--(1,2);
\draw [fill] (-0.5,1) circle [radius=0.1];
\draw [fill] (1,1) circle [radius=0.1];
\draw [fill] (2.5,1) circle [radius=0.1];
\draw [fill] (1,0) circle [radius=0.1];
\draw [fill] (-0.5,2) circle [radius=0.1];
\draw [fill] (1,2) circle [radius=0.1];
\draw [fill] (2.5,2) circle [radius=0.1];
\draw [fill] (1,3) circle [radius=0.1];
\draw [fill] (1.75,0.5) circle [radius=0.1];

\end{tikzpicture}}
}
\subfigure[$9_L$]
{{\scalefont{0.8}
\begin{tikzpicture}[scale=0.6]
\draw (1,0)--(0,1)--(-0.5,2)--(1,3);
\draw (0,1)--(0.5,2)--(1,3);
\draw (1,0)--(1,1)--(0.5,2);
\draw (1,1)--(1.5,2)--(1,3);
\draw (2,1)--(1.5,2);
\draw (1,0)--(2,1)--(2.5,2)--(1,3);
\draw [fill] (-0.5,2) circle [radius=0.1];
\draw [fill] (0,1) circle [radius=0.1];
\draw [fill] (2,1) circle [radius=0.1];
\draw [fill] (0.5,2) circle [radius=0.1];
\draw [fill] (-0.5,2) circle [radius=0.1];
\draw [fill] (1,1) circle [radius=0.1];
\draw [fill] (1.5,2) circle [radius=0.1];
\draw [fill] (1,3) circle [radius=0.1];
\draw [fill] (2.5,2) circle [radius=0.1];
\draw [fill] (1,0) circle [radius=0.1];
\end{tikzpicture}}
}
\subfigure[$9_M$]
{{\scalefont{0.8}
\begin{tikzpicture}[scale=0.6]
\draw (1,0)--(0,1)--(-0.5,2)--(1,3);
\draw (0.5,2)--(1,3);
\draw (1,0)--(1,1)--(0.5,2);
\draw (1,1)--(1.5,2)--(1,3);
\draw (2,1)--(1.5,2)--(0,1);
\draw (1,0)--(2,1)--(2.5,2)--(1,3);
\draw [fill] (-0.5,2) circle [radius=0.1];
\draw [fill] (0,1) circle [radius=0.1];
\draw [fill] (2,1) circle [radius=0.1];
\draw [fill] (0.5,2) circle [radius=0.1];
\draw [fill] (-0.5,2) circle [radius=0.1];
\draw [fill] (1,1) circle [radius=0.1];
\draw [fill] (1.5,2) circle [radius=0.1];
\draw [fill] (1,3) circle [radius=0.1];
\draw [fill] (2.5,2) circle [radius=0.1];
\draw [fill] (1,0) circle [radius=0.1];

\end{tikzpicture}}
}
\caption{The 13 semilattice structures that need to be studied separately.}\label{fig:tapaukset}
\end{figure}

\subsection*{Inserting a new maximal element}

Let us consider the first four structures $9_A$, $9_B$, $9_C$ and $9_D$. These four structures can also be obtained by using the methods in \cite{AA17}, but by using a slightly more general approach to construct them.

\begin{itemize}
\item Start from an 8 element set $S_8$ whose LCM matrix $[S_8]$ is singular. We may use, e.g., the set $S_8=\{1,2,3,5,66,70,255,39270\}$ illustrated in Figure \ref{Fig:singularset} (see also \cite{MHM20}, Example 5.3). The figure also displays the indexing of each element $x_i\in S_8$.
\item Choose an integer $a$ such that $\gcd(x_i,a)=1$ for all $x_i\in S_8$. For example in the case of the above mentioned set we may choose $a=13$.
\item In the case $9_A$ define $x_9=a x_1$ and $S_9=S_8\cup \{x_9\}$.
\item In the case $9_B$ define $x_9=ax_4$ and $S_9=S_8\cup \{x_9\}$.
\item In the case $9_C$ define $x_9=ax_7$ and $S_9=S_8\cup \{x_9\}$.
\item In the case $9_D$ define $x_9=ax_8$ and $S_9=S_8\cup \{x_9\}$.
\item In any case we have $\mu_{S_8}(x_i,x_8)=\mu_{S_9}(x_i,x_8)$ for all $i=1,2,\ldots,8$ and $x_9\,\nmid\,x_8$, which implies that $\mu_{S_8}(x_9,x_8)=0$. Thus $\Psi_{S_9}(x_8)=\Psi_{S_8}(x_8)=0$ and we may conclude that the matrix $[S_9]$ is singular.
\end{itemize}

\begin{figure}[h!]
\centering
{\scalefont{0.8}
\begin{tikzpicture}[scale=1]
\draw (1,0)--(-0.5,1)--(-0.5,2)--(1,3);
\draw (1,0)--(1,1)--(-0.5,2);
\draw (-0.5,1)--(1,2)--(1,2.9);
\draw (1,0)--(2.5,1)--(2.5,2)--(1,3);
\draw (1,1)--(2.5,2);
\draw (2.5,1)--(1,2);
\draw [fill] (-0.5,1) circle [radius=0.1];
\draw [fill] (1,1) circle [radius=0.1];
\draw [fill] (2.5,1) circle [radius=0.1];
\draw [fill] (1,0) circle [radius=0.1];
\draw [fill] (-0.5,2) circle [radius=0.1];
\draw [fill] (1,2) circle [radius=0.1];
\draw [fill] (2.5,2) circle [radius=0.1];
\draw [fill] (1,3) circle [radius=0.1];
\node [right] at (1.1,0) {$x_1=1$};
\node [right] at (-0.4,1) {$x_2=2$};
\node [right] at (1.1,1) {$x_3=3$};
\node [right] at (2.6,1) {$x_4=5$};
\node [right] at (-0.4,2) {$x_5=66$};
\node [right] at (1.1,2) {$x_6=70$};
\node [right] at (2.6,2) {$x_7=255$};
\node [right] at (1.1,3) {$x_8=39270$};
\end{tikzpicture}}
\caption{An example of a GCD closed set with eight elements such that the corresponding LCM matrix is singular.}\label{Fig:singularset}
\end{figure}
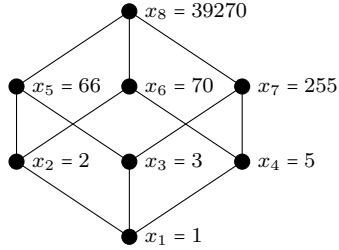

\subsection*{Inserting a new minimum element}

Let us now consider the semilattice structure $9_E$. This kind of approach was introduced by Hong \cite{Hong99} in 1999 as it enables one to construct a singular LCM matrix of arbitrary size $n\geq 9$. Also the method used by Alt{\i}n{\i}\c{s}{\i}k and Alt{\i}nta\c{s} finds this semilattice structure possible. As we did earlier, we are going to address this special case from our lattice theoretic point of view. 

\begin{itemize}
\item Again we start from any 8 element set $S_8$ for which the matrix $[S_8]$ is singular.
\item Choose an arbitrary integer $a>1$.
\item Define $x_1'=a$, $x_{i+1}'=a x_i$ for all $i=1,2,\ldots,8$ and $S_9=\{x_1',x_2',\ldots,x_8',x_9'\}$
\item We have $\mu_{S_8}(x_i,x_8)=\mu_{S_9}(x_{i+1}',x_9')$ for all $i=1,2,\ldots,8$ and $\mu_{S_9}(x_1',x_9')=0$, which implies that $\Psi_{S_9}(x_9)=\frac{1}{a}\Psi_{S_8}(x_8)=0$. Hence the matrix $[S_9]$ is singular.
\end{itemize}

\subsection*{Inserting a new element between two comparable elements}

Let us now focus on the semilattices $9_F$, $9_G$ and $9_H$. The method of Alt{\i}n{\i}\c{s}{\i}k and Alt{\i}nta\c{s} finds all three structures possible. We shall again provide a slightly more general solution.

\begin{itemize}
\item Start from any 8 element set $S_8$ such that $x_1=1$ and the matrix $[S_8]$ is singular.
\item Index the elements so that $x_5$ is an upper bound for $x_2$ and $x_3$.
\item It can be shown that $\text{lcm}(x_2,x_3)\,|\,x_5$, but $x_5>\mathrm{lcm}(x_2,x_3)$, see e.g. \cite[Theorem 2.4]{AA17}.
\item Let us denote $x_5=a\text{lcm}(x_2,x_3)$, where $a>1$.
\item In many (possibly in all) known examples the number $a$ turns out to be a prime number that does not divide either $x_2$ or $x_3$, which means that $\gcd(a,x_i)=1$ for all $x_i\in S_8\setminus\{x_8\}$. In the unlikely event when it would not be possible to apply this method to the element $x_5$ we can always try again with the elements $x_6$ and $x_7$ and redraw the lattice by putting either of these elements on the left side of the cube. In the example presented in Figure \ref{Fig:singularset} this method can be applied to any of the elements $x_5$, $x_6$ or $x_7$ since we have $x_5=66=11\cdot\mathrm{lcm}(2,3)$, $x_6=70=7\cdot\mathrm{lcm}(2,5)$ and $x_7=255=17\cdot\mathrm{lcm}(3,5)$.
\item In the case $9_F$  define $x_1'=x_1=1$, $x_2'=a$ and $x_i'=x_{i-1}$ for $i=3,4,\ldots,8,9$. We have $\mu_{S'}(x_2',x_9')=0$ and therefore $\Psi_{S_9}(x_9')=\Psi_{S_8}(x_8)=0$ and thus $[S_9]$ is singular.
\item In the case $9_G$ define $x_i'=x_i$ for $i=1,2,3,4$, $x_5'=ax_2$ and $x_i'=x_{i-1}$ for $i=6,7,8,9$. Also in this case $\mu_{S_9}(x_5',x_9')=0$ and furthermore $\Psi_{S_9}(x_9')=\Psi_{S_8}(x_8)=0$ making the matrix $[S_9]$ is singular.
\item In the case $9_H$ define $x_i'=x_i$ for $i=1,2,3,4,6,7$, $x_5'=\mathrm{lcm}(x_2,x_3)$, $x_8'=x_5=a\mathrm{lcm}(x_2,x_3)$ and $x_9'=x_8$. Again $\mu_{S_9}(x_5',x_9')=0$, $\Psi_{S_9}(x_9')=\Psi_{S_8}(x_8)=0$ and the matrix $[S_9]$ is singular.
\end{itemize}

\subsection*{Two irreducible cases}

Let us now take the semilattices $9_I$ and $9_J$ into consideration. 

\begin{itemize}
\item With the method used by Alt{\i}n{\i}\c{s}{\i}k and Alt{\i}nta\c{s} it is impossible to obtain these two structures.
\item Nevertheless, both of these structures are actually possible altough they cannot be reduced to the $8$ element case.
\item In 1997 Haukkanen, Wang and Sillanpää \cite{HauWanSil} disproved the Bourque-Ligh conjecture by providing the counterexample $S=\{1, 2, 3, 4, 5, 6, 10, 45, 180\}$, which belongs to the class $9_I$.
\item Also the structure $9_J$ is possible, for example the set $$S=\{1,5,11,17,19,748,1463,2907,4476780\}$$ is of this type and its LCM matrix is singular. This example is illustrated in Figure \ref{Fig:singularset2}. The singularity of the matrix $[S_9]$ follows from the calculation
\begin{align*}
\Psi_{S_{9}}(x_9)&=\frac{1}{4476780}-\frac{1}{2907}-\frac{1}{1463}-\frac{1}{748}-\frac{1}{255}+\frac{1}{19}+\frac{1}{17}+\frac{1}{11}+\frac{0}{1}=0.
\end{align*}
\end{itemize}

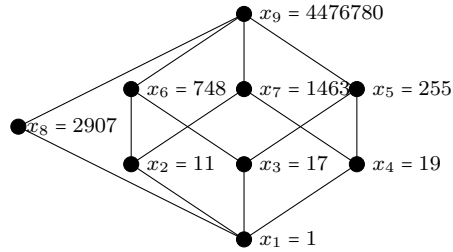
\begin{figure}[h!]
\centering
{\scalefont{0.8}
\begin{tikzpicture}[scale=1]
\draw (1,0)--(-0.5,1)--(-0.5,2)--(1,3);
\draw (1,0)--(1,1)--(-0.5,2);
\draw (-0.5,1)--(1,2)--(1,2.9);
\draw (1,0)--(2.5,1)--(2.5,2)--(1,3);
\draw (1,1)--(2.5,2);
\draw (2.5,1)--(1,2);
\draw (1,0)--(-2,1.5)--(1,3);
\draw [fill] (-2,1.5) circle [radius=0.1];
\draw [fill] (-0.5,1) circle [radius=0.1];
\draw [fill] (1,1) circle [radius=0.1];
\draw [fill] (2.5,1) circle [radius=0.1];
\draw [fill] (1,0) circle [radius=0.1];
\draw [fill] (-0.5,2) circle [radius=0.1];
\draw [fill] (1,2) circle [radius=0.1];
\draw [fill] (2.5,2) circle [radius=0.1];
\draw [fill] (1,3) circle [radius=0.1];
\node [right] at (1.1,0) {$x_1=1$};
\node [right] at (-0.4,1) {$x_2=11$};
\node [right] at (1.1,1) {$x_3=17$};
\node [right] at (2.6,1) {$x_4=19$};
\node [right] at (-0.4,2) {$x_6=748$};
\node [right] at (1.1,2) {$x_7=1463$};
\node [right] at (2.6,2) {$x_5=255$};
\node [right] at (-2,1.5) {$x_8=2907$};
\node [right] at (1.1,3) {$x_9=4476780$};
\end{tikzpicture}}
\caption{An example of a GCD closed set belonging to the category $9_J$ such that the corresponding LCM matrix is singular.}\label{Fig:singularset2}
\end{figure}

\subsection*{One negative example of inserting a new element into the middle}

We shall now focus on the structure $9_K$. It would essentially belong to the earlier category ``inserting a new element between two comparable elements'', but there are good reasons why we should study it separately.

\begin{itemize}
\item Alt{\i}n{\i}\c{s}{\i}k and Alt{\i}nta\c{s} attempted to find an example that would belong to this class, but they were not successful. They ended up conjecturing that this structure is not possible at all, see \cite[Conjecture 4.1]{AA17}.
\item Since $\mu_{S_9}(x_2,x_9)=0$, there appears to be no difference between this and the corresponding 8 element case.
\item However, it turns out that the existence of the element $x_2$ changes the situation quite drastically.
\item Without loss of generality we may assume that $x_1=1$ and that the elements $x_2$ and $x_3$ are named as shown in Figure \ref{fig:9K}. Because we have $x_1\,|\,x_2\,|\,x_3$, the element $x_3$ cannot be a prime number. We have
\[
\Psi_{S_9}(x_9)=\underbrace{\frac{1}{x_9}-\frac{1}{x_8}-\frac{1}{x_7}-\frac{1}{x_6}}_{<0}+\underbrace{\frac{1}{x_5}+\frac{1}{x_4}+\frac{1}{x_3}-\frac{1}{x_1}}_{:=K<0}<0,
\]
since
\[
K\leq -1+\max\left(\frac{1}{2}+\frac{1}{3}+\frac{1}{25},\frac{1}{4}+\frac{1}{3}+\frac{1}{5},\frac{1}{2}+\frac{1}{9}+\frac{1}{5}\right)<0.
\]
This shows that the class $9_K$ is not possible, which also solves the Conjecture 4.1 by Alt{\i}n{\i}\c{s}{\i}k and Alt{\i}nta\c{s} (the fact that $x_3$ cannot be a prime number also prevents the LCM matrix of the set $S_8':=S_9\setminus\{x_2\}$ to be singular since the values $\Psi_{S_9}(x_9)$ and $\Psi_{S_8'}(x_8')$ are exactly the same).
\end{itemize}

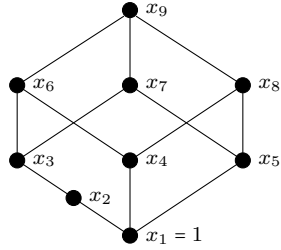
\begin{figure}[h!]
\begin{center}
{{\scalefont{0.8}
\begin{tikzpicture}[scale=1]
\draw (1,0)--(-0.5,1)--(-0.5,2)--(0.92,2.92);
\draw (1,0)--(1,1)--(-0.5,2);
\draw (-0.5,1)--(1,2)--(1,2.9);
\draw (1,0)--(2.5,1)--(2.5,2)--(1.08,2.92);
\draw (1,1)--(2.5,2);
\draw (2.5,1)--(1,2);
\draw [fill] (-0.5,1) circle [radius=0.1];
\draw [fill] (1,1) circle [radius=0.1];
\draw [fill] (2.5,1) circle [radius=0.1];
\draw [fill] (1,0) circle [radius=0.1];
\draw [fill] (-0.5,2) circle [radius=0.1];
\draw [fill] (1,2) circle [radius=0.1];
\draw [fill] (2.5,2) circle [radius=0.1];
\draw [fill] (1,3) circle [radius=0.1];
\draw [fill] (0.25,0.5) circle [radius=0.1];
\node [right] at (-0.5,1) {\ $x_3$};
\node [right] at (0.25,0.5) {\ $x_2$};
\node [right] at (1,0) {\ $x_1=1$};
\node [right] at (1,1) {\ $x_4$};
\node [right] at (2.5,1) {\ $x_5$};
\node [right] at (-0.5,2) {\ $x_6$};
\node [right] at (1,2) {\ $x_7$};
\node [right] at (2.5,2) {\ $x_8$};
\node [right] at (1,3) {\ $x_9$};
\end{tikzpicture}}
}
\end{center}
\caption{The structure $9_K$, where the different elements $x_i$ have been specified.}\label{fig:9K}
\end{figure}

\subsection*{Two more impossible cases}

Finally, we are left with the structures $9_L$ and $9_M$ shown in Figure \ref{fig:ljam}.

\begin{itemize}
\item In both of these structures every element except for the maximum element $x_9$ generates a double-chain set, which means that we only need to check the value of $\Psi_S(x_9)$.
\item In both cases we have
\begin{align*}
\Psi_S(x_9)&=\frac{1}{x_9}-\frac{1}{x_8}-\frac{1}{x_7}-\frac{1}{x_6}-\frac{1}{x_5}+\frac{1}{x_4}+\frac{1}{x_3}+\frac{1}{x_2}+\frac{0}{x_1}\\
&=\underbrace{\frac{1}{x_9}}_{>0}+\underbrace{\left(\frac{1}{x_4}-\frac{1}{x_7}-\frac{1}{x_8}\right)}_{>0}+\underbrace{\left(\frac{1}{x_3}-\frac{1}{x_6}\right)}_{>0}+\underbrace{\left(\frac{1}{x_2}-\frac{1}{x_5}\right)}_{>0}>0,
\end{align*}
which implies that the matrix $[S_9]$ has to be nonsingular.
\end{itemize}

\setcounter{subfigure}{11}

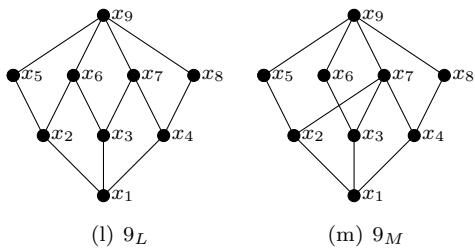
\begin{figure}[h!]
\begin{center}
\subfigure[$9_L$]
{{\scalefont{0.8}
\begin{tikzpicture}[scale=0.8]
\draw (1,0)--(0,1)--(-0.5,2)--(1,3);
\draw (0,1)--(0.5,2)--(1,3);
\draw (1,0)--(1,1)--(0.5,2);
\draw (1,1)--(1.5,2)--(1,3);
\draw (2,1)--(1.5,2);
\draw (1,0)--(2,1)--(2.5,2)--(1,3);
\draw [fill] (-0.5,2) circle [radius=0.1];
\draw [fill] (0,1) circle [radius=0.1];
\draw [fill] (2,1) circle [radius=0.1];
\draw [fill] (0.5,2) circle [radius=0.1];
\draw [fill] (-0.5,2) circle [radius=0.1];
\draw [fill] (1,1) circle [radius=0.1];
\draw [fill] (1.5,2) circle [radius=0.1];
\draw [fill] (1,3) circle [radius=0.1];
\draw [fill] (2.5,2) circle [radius=0.1];
\draw [fill] (1,0) circle [radius=0.1];
\node [right] at (0,1) {$x_2$};
\node [right] at (2,1) {$x_4$};
\node [right] at (0.5,2) {$x_6$};
\node [right] at (-0.5,2) {$x_5$};
\node [right] at (1,1) {$x_3$};
\node [right] at (1.5,2) {$x_7$};
\node [right] at (1,3) {$x_9$};
\node [right] at (2.5,2) {$x_8$};
\node [right] at (1,0) {$x_1$};
\end{tikzpicture}}
}
\subfigure[$9_M$]
{{\scalefont{0.8}
\begin{tikzpicture}[scale=0.8]
\draw (1,0)--(0,1)--(-0.5,2)--(1,3);
\draw (0.5,2)--(1,3);
\draw (1,0)--(1,1)--(0.5,2);
\draw (1,1)--(1.5,2)--(1,3);
\draw (2,1)--(1.5,2)--(0,1);
\draw (1,0)--(2,1)--(2.5,2)--(1,3);
\draw [fill] (-0.5,2) circle [radius=0.1];
\draw [fill] (0,1) circle [radius=0.1];
\draw [fill] (2,1) circle [radius=0.1];
\draw [fill] (0.5,2) circle [radius=0.1];
\draw [fill] (-0.5,2) circle [radius=0.1];
\draw [fill] (1,1) circle [radius=0.1];
\draw [fill] (1.5,2) circle [radius=0.1];
\draw [fill] (1,3) circle [radius=0.1];
\draw [fill] (2.5,2) circle [radius=0.1];
\draw [fill] (1,0) circle [radius=0.1];
\node [right] at (0,1) {$x_2$};
\node [right] at (2,1) {$x_4$};
\node [right] at (0.5,2) {$x_6$};
\node [right] at (-0.5,2) {$x_5$};
\node [right] at (1,1) {$x_3$};
\node [right] at (1.5,2) {$x_7$};
\node [right] at (1,3) {$x_9$};
\node [right] at (2.5,2) {$x_8$};
\node [right] at (1,0) {$x_1$};
\end{tikzpicture}}}
\end{center}
\caption{Illustrations of the structures $9_L$ and $9_M$, where the choices for different elements $x_i$ are also displayed.}\label{fig:ljam}
\end{figure}

\subsection*{Some conclusions}

We were able to completely solve the 9 element case and show that there are exactly $10$ possible semilattice structures that can be used to construct a GCD closed 9-element set whose LCM matrix is singular. Based on this investigation we can also say that if $S$ is a GCD closed set with $9$ elements and the matrix $[S]$ is singular, then the semilattice $(S,|)$ has the cube as its subsemilattice. Alt{\i}n{\i}\c{s}{\i}k and Alt{\i}nta\c{s} actually conjecture that this is the case with all sets $S$ whose LCM matrix is singular, and we shall consider this conjecture in the next section. In theory it would now be possible to proceed to the next case and try to find all possible semilattice structures with $10$ elements such that there is a corresponding set of positive integers whose LCM matrix is singular. However, as mentioned earlier there are $37622$ meet semilattice structures with $10$ elements and it turns out that there are $166$ structures where at least one element does not generate a double-chain set within the structure. In other words, solving the $10$ element case would require us to either do a huge amount of tedious calculations or to develop more sophisticated methods than we currently have at our disposal.

\section{A couple of interesting examples with more than 9 elements}

In many cases the $8$ element cube semilattice can be found as a subsemilattice of the semilattice $(S,|)$ when the LCM matrix is singular, and this may even happen in quite unexpected ways. We shall illustrate this with a couple of examples.

\setcounter{subfigure}{0}

\begin{example}\label{example}
Let us consider the GCD closed sets
\[
S_{13}=\{1,2,3,13,23,25,41,75,369,533,6877,16675,3679538850\},
\]
\[
S_{14}=\{1,2,3,6,7,11,13,19,56,147,209,1859,196859,33105384312\}
\]
and
\[
S_{16}=\{1,2,3,5,7,14,20,35,54,231,255,1820,45738,137445,39308760,3029801294520\}.
\]
The Hasse diagrams of these three semilattices and also the indexing of different elements $x_i$ in each case are shown in Figures \ref{fig:examplea}, \ref{fig:exampleb} and \ref{fig:examplec}.

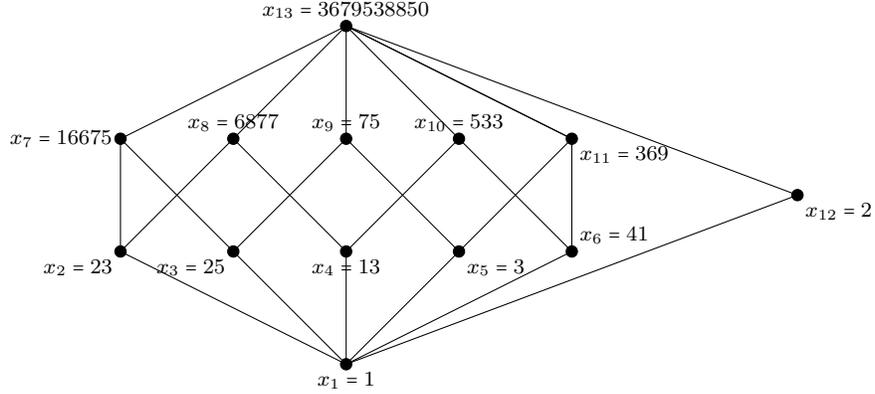
\begin{figure}[h!]
\begin{center}
{\scalefont{0.8}
\begin{tikzpicture}[scale=1.5]
\draw (0,0)--(-2,1)--(-1,2)--(0,3)--(-2,2)--(-2,1);
\draw (-2,2)--(-1,1)--(0,0)--(0,1)--(-1,2);
\draw (0,0)--(1,1)--(0,2)--(-1,1);
\draw (0,0)--(2,1)--(2,2)--(0,3)--(1,2)--(2,1);
\draw (1,1)--(2,2)--(0,3);
\draw (0,0)--(4,1.5)--(0,3);
\draw (0,2)--(0,3);
\draw (0,1)--(1,2);
\draw [fill] (0,0) circle [radius=0.05];
\draw [fill] (-2,1) circle [radius=0.05];
\draw [fill] (-1,1) circle [radius=0.05];
\draw [fill] (0,1) circle [radius=0.05];
\draw [fill] (1,1) circle [radius=0.05];
\draw [fill] (2,1) circle [radius=0.05];
\draw [fill] (-2,2) circle [radius=0.05];
\draw [fill] (-1,2) circle [radius=0.05];
\draw [fill] (0,2) circle [radius=0.05];
\draw [fill] (1,2) circle [radius=0.05];
\draw [fill] (2,2) circle [radius=0.05];
\draw [fill] (4,1.5) circle [radius=0.05];
\draw [fill] (0,3) circle [radius=0.05];
\node [below] at (0,0) {$x_1=1$};
\node [below left] at (-2,1) {$x_2=23$};
\node [below left] at (-1,1) {$x_3=25$};
\node [below] at (0,1) {$x_4=13$};
\node [below right] at (1,1) {$x_5=3$};
\node [above right] at (2,1) {$x_6=41$};
\node [left] at (-2,2) {$x_7=16675$};
\node [above] at (-1,2) {$x_8=6877$};
\node [above] at (0,2) {$x_9=75$};
\node [above] at (1,2) {$x_{10}=533$};
\node [below right] at (2,2) {$x_{11}=369$};
\node [below right] at (4,1.5) {$x_{12}=2$};
\node [above] at (0,3) {$x_{13}=3679538850$};
\end{tikzpicture}}
\end{center}
\caption{Illustrations of the three semilattice structure $S_{13}$ of Example \ref{example}. The choices for different elements $x_i$ are also shown in the picture.}\label{fig:examplea}
\end{figure}

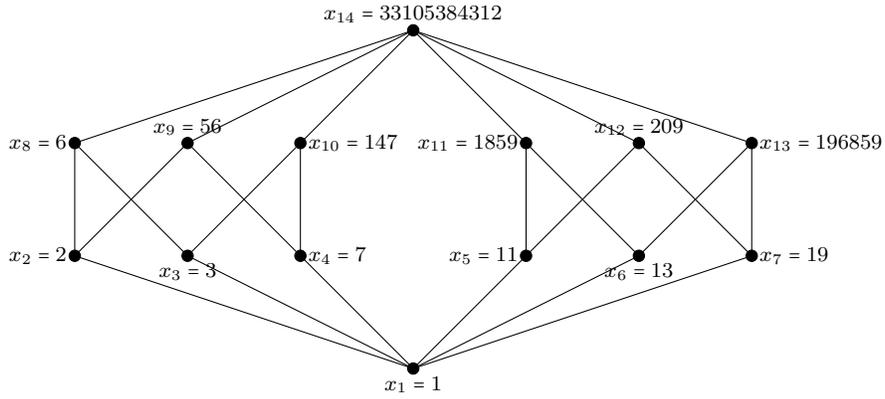
\begin{figure}[h!]
\begin{center}
{\scalefont{0.8}
\begin{tikzpicture}[scale=1.5]
\draw (0,0)--(3,1)--(3,2)--(0,3);
\draw (0,0)--(2,1)--(1,2)--(0,3);
\draw (0,0)--(1,1)--(2,2)--(0,3);
\draw (3,2)--(2,1);
\draw (2,2)--(3,1);
\draw (-1,1)--(-1,2);
\draw (0,0)--(-3,1)--(-3,2)--(0,3);
\draw (0,0)--(-2,1)--(-1,2)--(0,3);
\draw (0,0)--(-1,1)--(-2,2)--(0,3);
\draw (-3,2)--(-2,1);
\draw (-2,2)--(-3,1);
\draw (1,1)--(1,2);
\draw [fill] (0,0) circle [radius=0.05];
\draw [fill] (-3,1) circle [radius=0.05];
\draw [fill] (-2,1) circle [radius=0.05];
\draw [fill] (-1,1) circle [radius=0.05];
\draw [fill] (1,1) circle [radius=0.05];
\draw [fill] (2,1) circle [radius=0.05];
\draw [fill] (3,1) circle [radius=0.05];
\draw [fill] (-3,2) circle [radius=0.05];
\draw [fill] (-2,2) circle [radius=0.05];
\draw [fill] (-1,2) circle [radius=0.05];
\draw [fill] (1,2) circle [radius=0.05];
\draw [fill] (2,2) circle [radius=0.05];
\draw [fill] (3,2) circle [radius=0.05];
\draw [fill] (0,03) circle [radius=0.05];
\node [left] at (-3,1) {$x_2=2$};
\node [below] at (-2,1) {$x_3=3$};
\node [right] at (-1,1) {$x_4=7$};
\node [left] at (1,1) {$x_5=11$};
\node [below] at (2,1) {$x_6=13$};
\node [right] at (3,1) {$x_7=19$};
\node [left] at (-3,2) {$x_8=6$};
\node [above] at (-2,2) {$x_9=56$};
\node [below] at (0,0) {$x_1=1$};
\node [right] at (-1,2) {$x_{10}=147$};
\node [left] at (1,2) {$x_{11}=1859$};
\node [above] at (2,2) {$x_{12}=209$};
\node [right] at (3,2) {$x_{13}=196859$};
\node [above] at (0,3) {$x_{14}=33105384312$};
\end{tikzpicture}}
\end{center}
\caption{Illustrations of the semilattice structure $S_{14}$ considered in Example \ref{example}. The choices for different elements $x_i$ are also shown in the picture.}\label{fig:exampleb}
\end{figure}

\begin{figure}[h!]
\begin{center}
{\scalefont{0.8}
\begin{tikzpicture}[scale=1.5]
\draw (0,0)--(-1,1)--(-1,2)--(0,3)--(4,4)--(5,3)--(5,2)--(4,1)--(0,0);
\draw (0,0)--(0,1)--(-1,2)--(3,3)--(4,2)--(0,1)--(1,2)--(5,3)--(4,2)--(4,1);
\draw (0,1)--(1,2)--(5,3)--(4,2)--(0,1);
\draw (0,3)--(1,2)--(5,3);
\draw (4,4)--(3,3)--(3,2)--(-1,1);
\draw (-1,1)--(-1,2);
\draw (0,0)--(1,1)--(5,2)--(4,3)--(0,2)--(-1,1);
\draw (1,2)--(1,1)--(0,2);
\draw (4,1)--(3,2)--(4,3)--(4,4);
\draw (0,3)--(0,2);
\draw [fill] (0,0) circle [radius=0.05];
\draw [fill] (-1,1) circle [radius=0.05];
\draw [fill] (0,1) circle [radius=0.05];
\draw [fill] (1,1) circle [radius=0.05];
\draw [fill] (4,1) circle [radius=0.05];
\draw [fill] (-1,2) circle [radius=0.05];
\draw [fill] (0,2) circle [radius=0.05];
\draw [fill] (1,2) circle [radius=0.05];
\draw [fill] (3,2) circle [radius=0.05];
\draw [fill] (4,2) circle [radius=0.05];
\draw [fill] (5,2) circle [radius=0.05];
\draw [fill] (0,3) circle [radius=0.05];
\draw [fill] (3,3) circle [radius=0.05];
\draw [fill] (4,3) circle [radius=0.05];
\draw [fill] (5,3) circle [radius=0.05];
\draw [fill] (4,4) circle [radius=0.05];
\node [below left] at (-1,1) {$x_2=2$};
\node [below] at (0,1) {$x_3=3$};
\node [below right] at (1,1) {$x_4=5$};
\node [below right] at (4,1) {$x_5=7$};
\node [left] at (-1,2) {$x_6=54$};
\node [above] at (0,2) {$x_7=20$};
\node [below right] at (1,2) {$x_8=255$};
\node [above left] at (3,2) {$x_9=14$};
\node [below] at (0,0) {$x_1=1$};
\node [above] at (4,2) {$x_{10}=231$};
\node [below right] at (5,2) {$x_{11}=35$};
\node [above left] at (0,3) {$x_{12}=39308760$};
\node [above left] at (3,3) {$x_{13}=45738$};
\node [above] at (4,3) {$x_{14}=1820$};
\node [above right] at (5,3) {$x_{15}=137445$};
\node [above] at (4,4) {$x_{16}=3029801294520$};
\end{tikzpicture}}
\end{center}
\caption{Illustration of the semilattice structure $S_{16}$ of Example \ref{example}. The choices for different elements $x_i$ are also shown in the picture.}\label{fig:examplec}
\end{figure}
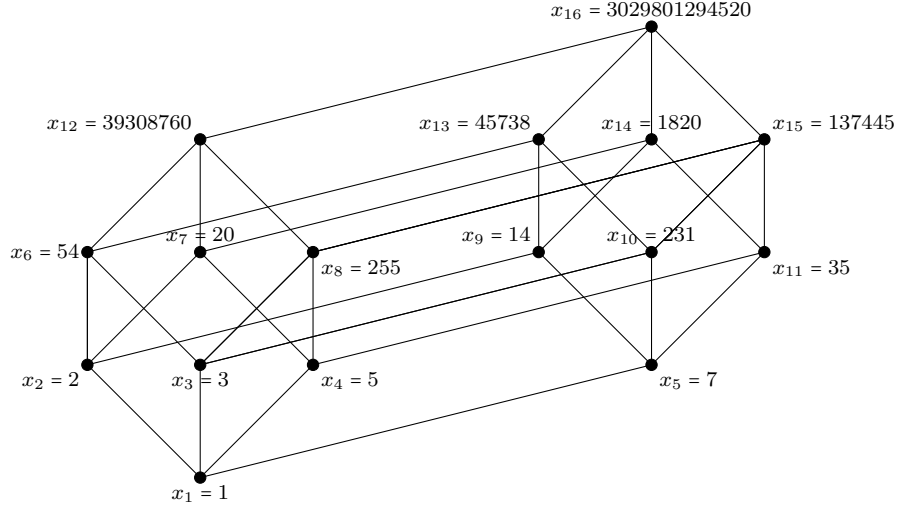

The matrices $[S_{13}]$, $[S_{14}]$ and $[S_{16}]$ are all singular since we have
\begin{align*}
\Psi_{S_{13}}(x_{13})&=\frac{1}{3679538850}-\frac{1}{16675}-\frac{1}{6877}-\frac{1}{533}-\frac{1}{369}-\frac{1}{75}-\frac{1}{2}\\
&\qquad\qquad+\frac{1}{41}+\frac{1}{25}+\frac{1}{23}+\frac{1}{13}+\frac{1}{3}+\frac{0}{1}=0,
\end{align*}
\begin{align*}
\Psi_{S_{14}}(x_{14})&=\frac{1}{33105384312}-\frac{1}{196859}-\frac{1}{1859}-\frac{1}{209}-\frac{1}{147}-\frac{1}{56}-\frac{1}{6}\\
&\qquad\qquad+\frac{1}{19}+\frac{1}{13}+\frac{1}{11}+\frac{1}{7}+\frac{1}{3}+\frac{1}{2}-\frac{1}{1}=0
\end{align*}
and
\begin{align*}
\Psi_{S_{16}}(x_{16})&=\frac{1}{3029801294520}-\frac{1}{39308760}-\frac{1}{137445}-\frac{1}{45738}-\frac{1}{1820}\\
&\qquad\qquad+\frac{1}{255}+\frac{1}{231}+\frac{1}{54}+\frac{1}{35}+\frac{1}{20}+\frac{1}{14}-\frac{1}{7}-\frac{1}{5}-\frac{1}{3}-\frac{1}{2}+\frac{1}{1}=0.
\end{align*}
First we should observe that the set $S_{13}$ does not have the $8$ element cube semilattice as its subsemilattice. This implies that Conjecture 4.2 by Alt{\i}n{\i}\c{s}{\i}k and Alt{\i}nta\c{s} \cite{AA17} does not hold in general for GCD closed sets that have at least $13$ elements (the conjecture remains still open in the cases when there are exactly $10$, $11$ or $12$ elements in the set $S$). On the other hand, in the sets $S_{14}$ and $S_{16}$ there are actually two different cube-type subsemilattices. In $S_{14}$ these two subsemilattices have the top and bottom elements in common whereas in $S_{16}$ the two subsemilattices are totally disjoint. One can only say that when it comes to GCD closed sets whose LCM matrix is singular, the semilattice structure of $(S,|)$ can sometimes take quite surprising forms.
\end{example}

\section{Lattice theoretic approach to Sun's conjecture}

According to Hong \cite{HongRfold}, in January 1997 (just before the Bouque-Ligh conjecture was solved) professor Qi Sun at Sichuan University had a private conversation with Hong and conjectured that if $S$ is a GCD closed set satisfying $\max_{x_i\in S}\{\omega(x_i)\}\leq 2$, then the LCM matrix $[S]$ is nonsingular ($\omega$ is the arithmetic function that counts the number of distinct prime divisors of the input element $x_i$). In 2005 the conjecture was proven right by Hong in the above mentioned article. This number theoretic proof is quite technical and requires the use of several auxiliary results. We are now going to apply our lattice theoretic method once again and give a new proof which is rather short, visually intuitive and easy to understand.

\begin{theorem}[Sun's conjecture]
Let $S=\{x_1,x_2,\ldots,x_n\}$ be a GCD closed set, where every element $x_i\in S$ has at most two distinct prime divisors. Then the LCM matrix $[S]$ is invertible. 
\end{theorem}

\begin{proof}
We only need to show that $\Psi_S(x_i)\neq 0$ for all $x_i\in S$. For the element $x_1$ we have $\Psi_S(x_1)=\frac{1}{x_1}>0$.
If $x_i$ covers exactly one element in $S$ (namely $x_j$), then we have
\[
\Psi_S(x_i)=\frac{1}{x_i}-\frac{1}{x_j}<0
\]
Suppose then that $x_i$ covers at least two elements in $S$. Because $\max_{x_i\in S}\{\omega(x_i)\}\leq 2$, all the divisors of $x_i$ are of the form $p^kq^l$, where $p,q\in \mathbb{P}$. Let $p^{a_1}q^{b_1}, p^{a_2}q^{b_2},\ldots,p^{a_r}q^{b_r}$ be the elements that are covered by $x_i$ in $S$. Without the loss of generality we may assume that $a_1<a_2<\cdots<a_r$. Because the elements $p^{a_k}q^{b_k}$ are incomparable, we must now have $b_1>b_2>\cdots>b_r$. Since $S$ is GCD closed, for any $k<l$ we have $\gcd(p^{a_k}q^{b_k},p^{a_l}q^{b_l})=p^{a_k}q^{b_l}\in S$. It should be noted that there may exist also other elements of type $p^aq^b$ in $S$ such that $p^aq^b\,|\,x_i$ and we have either $a_{k}<a<a_{k+1}$ for some $k\in\{1,\ldots,r-1\}$, $b_{l+1}<b<b_l$ for some $l\in\{1,\ldots,r-1\}$ or $a<a_1$ and $b<b_r$. However, in these possible cases we have $\mu_S(p^aq^b,x_i)=0$ by \cite[Lemma 3.2]{MHM15} and therefore these types of elements do not produce any nonzero terms to the function $\Psi_S(x_i)$ and also the existence of such elements does not affect on the calculation of the values $\mu_S(p^{a_k}q^{b_l},x_i)$, where $k,l\in\{1,\ldots,r\}$.

\begin{figure}
\centering
\begin{tikzpicture}[scale=0.9]
\draw (3,0)--(-2,5)--(3,6);
\draw (3,0)--(8,5)--(3,6);
\draw (2,1)--(6,5)--(3,6);
\draw (1,2)--(4,5)--(3,6);
\draw (4,1)--(0,5)--(3,6);
\draw (5,2)--(2,5)--(3,6);
\draw (6,3)--(4,5)--(3,6);
\draw (3,2)--(0,5)--(3,6);
\draw (0,3)--(2,5)--(3,6);
\draw (-1,4)--(0,5)--(3,6);
\draw (7,4)--(6,5);
\draw [fill] (3,0) circle [radius=0.1];
\draw [fill] (0,3) circle [radius=0.1];
\draw [fill] (3,4) circle [radius=0.1];
\draw [fill] (6,3) circle [radius=0.1];
\draw [fill] (2,1) circle [radius=0.1];
\draw [fill] (4,3) circle [radius=0.1];
\draw [fill] (1,2) circle [radius=0.1];
\draw [fill] (2,3) circle [radius=0.1];
\draw [fill] (4,1) circle [radius=0.1];
\draw [fill] (5,2) circle [radius=0.1];
\draw [fill] (3,2) circle [radius=0.1];
\draw [fill] (3,6) circle [radius=0.1];
\draw [fill] (-1,4) circle [radius=0.1];
\draw [fill] (0,5) circle [radius=0.1];
\draw [fill] (2,5) circle [radius=0.1];
\draw [fill] (3,4) circle [radius=0.1];
\draw [fill] (4,5) circle [radius=0.1];
\draw [fill] (5,4) circle [radius=0.1];
\draw [fill] (-2,5) circle [radius=0.1];
\draw [fill] (1,4) circle [radius=0.1];
\draw [fill] (7,4) circle [radius=0.1];
\draw [fill] (8,5) circle [radius=0.1];
\draw [fill] (6,5) circle [radius=0.1];
\node [above right] at (3,6) {$x_i$};
\node [right] at (-2,5) {$p^{a_1}q^{b_1}$};
\node [right] at (0,5) {$p^{a_2}q^{b_2}$};
\node [right] at (2,5) {$p^{a_3}q^{b_3}$};
\node [right] at (8,5) {$p^{a_r}q^{b_r}$};
\node [right] at (6,5) {$p^{a_{r-1}}q^{b_{r-1}}$};
\node [right] at (-1,4) {$p^{a_1}q^{b_2}$};
\node [right] at (1,4) {$p^{a_2}q^{b_3}$};
\node [right] at (3,4) {$p^{a_3}q^{b_4}$};
\node [right] at (7,4) {$p^{a_{r-1}}q^{b_r}$};
\node [right] at (0,3) {$p^{a_1}q^{b_3}$};
\node [right] at (2,3) {$p^{a_2}q^{b_4}$};
\node [right] at (1,2) {$p^{a_1}q^{b_4}$};
\end{tikzpicture}
\caption{Illustration of the structure of the set $\mathrm{meetcl}(C_S(x_i))\cup\{x_i\}$ that we use in the proof of Sun's conjecture.}\label{fig:Sun}
\end{figure}
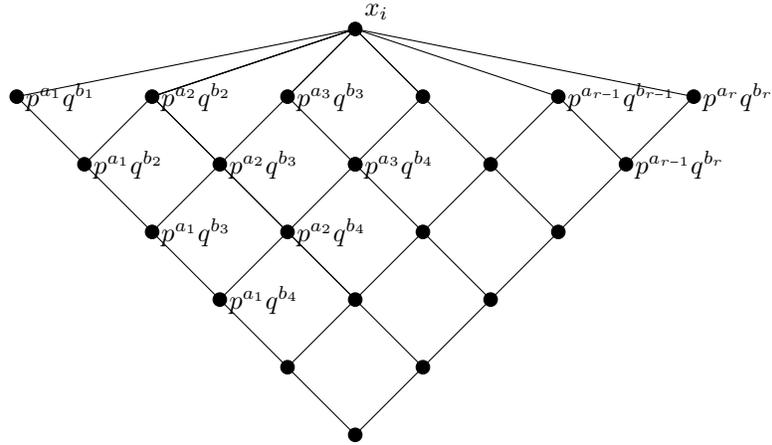

\begin{figure}
\centering
\begin{tikzpicture}[scale=0.9]
\draw (3,0)--(-2,5)--(3,6);
\draw (3,0)--(8,5)--(3,6);
\draw (2,1)--(6,5)--(3,6);
\draw (1,2)--(4,5)--(3,6);
\draw (4,1)--(0,5)--(3,6);
\draw (5,2)--(2,5)--(3,6);
\draw (6,3)--(4,5)--(3,6);
\draw (3,2)--(0,5)--(3,6);
\draw (0,3)--(2,5)--(3,6);
\draw (-1,4)--(0,5)--(3,6);
\draw (7,4)--(6,5);
\draw [fill] (3,0) circle [radius=0.1];
\draw [fill] (0,3) circle [radius=0.1];
\draw [fill] (3,4) circle [radius=0.1];
\draw [fill] (6,3) circle [radius=0.1];
\draw [fill] (2,1) circle [radius=0.1];
\draw [fill] (4,3) circle [radius=0.1];
\draw [fill] (1,2) circle [radius=0.1];
\draw [fill] (2,3) circle [radius=0.1];
\draw [fill] (4,1) circle [radius=0.1];
\draw [fill] (5,2) circle [radius=0.1];
\draw [fill] (3,2) circle [radius=0.1];
\draw [fill] (3,6) circle [radius=0.1];
\draw [fill] (-1,4) circle [radius=0.1];
\draw [fill] (0,5) circle [radius=0.1];
\draw [fill] (2,5) circle [radius=0.1];
\draw [fill] (3,4) circle [radius=0.1];
\draw [fill] (4,5) circle [radius=0.1];
\draw [fill] (5,4) circle [radius=0.1];
\draw [fill] (-2,5) circle [radius=0.1];
\draw [fill] (1,4) circle [radius=0.1];
\draw [fill] (7,4) circle [radius=0.1];
\draw [fill] (8,5) circle [radius=0.1];
\draw [fill] (6,5) circle [radius=0.1];
\node [above right] at (3,6) {$1$};
\node [right] at (-2,5) {$-1$};
\node [right] at (0,5) {$-1$};
\node [right] at (2,5) {$-1$};
\node [right] at (4,5) {$-1$};
\node [right] at (8,5) {$-1$};
\node [right] at (6,5) {$-1$};
\node [right] at (-1,4) {$1$};
\node [right] at (1,4) {$1$};
\node [right] at (3,4) {$1$};
\node [right] at (5,4) {$1$};
\node [right] at (7,4) {$1$};
\node [right] at (0,3) {$0$};
\node [right] at (2,3) {$0$};
\node [right] at (4,3) {$0$};
\node [right] at (6,3) {$0$};
\node [right] at (1,2) {$0$};
\node [right] at (3,2) {$0$};
\node [right] at (5,2) {$0$};
\node [right] at (2,1) {$0$};
\node [right] at (4,1) {$0$};
\node [right] at (3,0) {$0$};
\end{tikzpicture}
\caption{A figure that gives the Möbius function values $\mu_S(p^{a_k}q^{b_k},x_i)$ for different elements $p^{a_k}q^{b_k}\in\mathrm{meetcl}(C_S(x_i))\cup\{x_i\}$.}\label{fig:SunMöbius}
\end{figure}
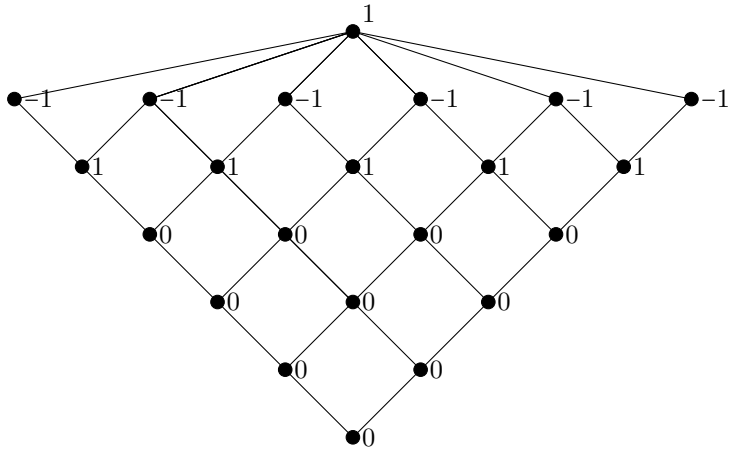

The structure of $\mathrm{meetcl}(C_S(x_i))\cup\{x_i\}$ is illustrated in Figure \ref{fig:Sun}. We also need to calculate the values $\mu_S(p^{a_k}q^{b_k},x_i)$ of the Möbius function on the set $S$, which is done in Figure \ref{fig:SunMöbius}. We now obtain
\begin{align*}
&\Psi_S(x_i)=\sum_{x_k\in S} \frac{\mu_S(x_k,x_i)}{x_k}=\frac{1}{x_i}-\sum_{j=1}^r \frac{1}{p^{a_j}q^{b_j}}+\sum_{j=1}^{r-1} \frac{1}{p^{a_j}q^{b_{j+1}}}\\
&=\frac{1}{x_i}+\left(\frac{1}{p^{a_1}q^{b_2}}-\frac{1}{p^{a_1}q^{b_1}}-\frac{1}{p^{a_2}q^{b_2}}\right)+\sum_{j=3}^{r-1} \left(\frac{1}{p^{a_{j-1}}q^{b_{j}}}-\frac{1}{p^{a_j}q^{b_{j}}}\right)\\
&=\frac{1}{x_i}+\frac{1}{p^{a_2}q^{b_1}}\underbrace{\left(p^{a_2-a_1}q^{b_1-b_2}-p^{a_2-a_1}-q^{b_1-b_2}\right)}_{=p^{a_2-a_1}(q^{b_1-b_2}-1)-q^{b_1-b_2}>0}+\sum_{j=3}^{r-1} \frac{1}{p^{a_j}q^{b_{j}}}\underbrace{\left(p^{a_j-a_{j-1}}-1\right)}_{>0}>0.
\end{align*}
Thus our proof is complete.
\end{proof}

\section{Extending the lattice theoretic method in order to study power LCM matrices}

The LCM matrix of the set $S$ with respect to the arithmetical function $N^\alpha$, where $N^\alpha(n)=n^\alpha$ for all $n\in\mathbb{Z}^+$, is the $n\times n$ matrix $[S]_{N^\alpha}$ with $(\mathrm{lcm}(x_i,x_j))^\alpha$ as its $ij$ entry. These type of matrices are known as power LCM matrices. If we wish to study the invertibility of the matrix $[S]_{N^\alpha}$ under the assumption that the set $S$ is GCD closed, we may use a similar approach as we did with the usual LCM matrices. By making use of the identity $(\gcd(x_i,x_j))^\alpha(\text{lcm}(x_i,x_j))^\alpha=x_i^\alpha x_j^\alpha$ we may decompose the matrix $[S]_{N^\alpha}$ as
\[
[S]_{N^{\alpha}}=\text{diag}(x_1^\alpha,x_2^\alpha,\ldots,x_n^\alpha)\left(\frac{1}{(\gcd(x_i,x_j))^\alpha}\right)\text{diag}(x_1^\alpha,x_2^\alpha,\ldots,x_n^\alpha).
\]
Next we use Möbius inversion in exactly the same way and define the function $\Psi_{S,\frac{1}{N^\alpha}}$ on $S$ as
\begin{equation}\label{eq:psialpha}
\Psi_{S,\frac{1}{N^\alpha}}(x_i)=\sum_{x_j\,|\,x_i}\frac{\mu_S(x_j,x_i)}{x_j^\alpha}.
\end{equation}
It should be noted that for the special case $\alpha=1$ we have $\Psi_{S,\frac{1}{N^\alpha}}=\Psi_{S}$.

Now the so-called reciprocal power GCD matrix $\left(\frac{1}{(\gcd(x_i,x_j))^\alpha}\right)$ may be written as
\[
\left(\frac{1}{(\gcd(x_i,x_j))^\alpha}\right)=E\,\text{diag}(\Psi_{S,\frac{1}{N^\alpha}}(x_1),\Psi_{S,\frac{1}{N^\alpha}}(x_2),\ldots,\Psi_{S,\frac{1}{N^\alpha}}(x_n))\,E^T,
\]
where $E=(e_{ij})$ is the $0,1$ incidence matrix defined in Section 2. Thereby we obtain the factorization
\begin{equation*}
[S]_{N^\alpha}=\Delta^\alpha E\Lambda E^T\Delta^\alpha=(\Delta^\alpha E)\Lambda(\Delta^\alpha E)^T,
\end{equation*}
where $\Delta^\alpha=\text{diag}(x_1^\alpha,x_2^\alpha,\ldots,x_n^\alpha)$ and $$\Lambda=\text{diag}(\Psi_{S,\frac{1}{N^\alpha}}(x_1),\Psi_{S,\frac{1}{N^\alpha}}(x_2),\ldots\Psi_{S,\frac{1}{N^\alpha}}(x_n)),$$ which leads us to the analogous conclusion that the power LCM matrix $[S]_{N^\alpha}$ is invertible if and only if $\Psi_{S,\frac{1}{N^\alpha}}(x_i)\neq0$ for all $i=1,2,\ldots,n.$ 

In 2002 Hong \cite{Hong02} raised the following conjecture concerning power LCM matrices.

\begin{conjecture}[cf. \cite{Hong02}, Conjecture]
If $t$ is a given positive integer, then there exists a positive integer $k(t)$ (depending only on $t$) such that the power LCM matrix $[(\mathrm{lcm}(x_i,x_j))^t]$ defined on any GCD closed set $S=\{x_1,x_2,\ldots,x_n\}$ is nonsingular for all $n\leq k(t)$. But for any $n\geq k(t)+1$ there exists a GCD closed set $S=\{x_1,x_2,\ldots,x_n\}$ such that the matrix $[(\mathrm{lcm}(x_i,x_j))^t]$ is singular.
\end{conjecture}

Known results for usual LCM matrices imply that $k(t)\geq 8$ for all $t\geq 2$, see \cite{Cao07}. The existence of $k(t)$ appears to be a highly nontrivial problem. However, we can use the cube semilattice to prove the following result that is closely related to the above conjecture.

\begin{theorem}\label{powerthm}
If $M\geq 1$ is an arbitrary real number, then there exists a GCD closed set $S=\{x_1,x_2,\ldots,x_n\}$ and real number $\alpha_0$ such that $\alpha_0> M$ and the power LCM matrix $[(\mathrm{lcm}(x_i,x_j))^{\alpha_0}]$ is singular.
\end{theorem} 

\begin{proof}
Let $k\geq 2$ be a positive integer such that $\displaystyle \frac{\ln k}{\ln 2}\geq M$. Let $p_1,p_2,\ldots,p_{k-1}$ be sufficiently large prime numbers (to be specified later). Let us consider the set
\[
S=\{1,2,3,5,6,10,15,2p_1,2p_2,\ldots, 2p_{k-1},30p_1p_2\cdots p_{k-1}\}
\]
with $n=k+7$ elements. We obtain the following semilattice structure shown in Figure \ref{fig:laajennettukuutio}. The corresponding Möbius function values $\mu_S(x_i,x_n)$ are also shown in the same figure.

\begin{figure}[h!]
\centering
\subfigure
{{\scalefont{0.8}
\begin{tikzpicture}[scale=1]
\draw (1,0)--(-0.5,1)--(-0.5,2)--(0.92,2.92);
\draw (1,0)--(1,1)--(-0.5,2);
\draw (-0.5,1)--(1,2)--(1,2.9);
\draw (1,0)--(2.5,1)--(2.5,2)--(1.08,2.92);
\draw (1,1)--(2.5,2);
\draw (-0.5,1)--(-2,2)--(1,3);
\draw (-0.5,1)--(-3.5,2)--(1,3);
\draw (-0.5,1)--(-6,2)--(1,3);
\draw (2.5,1)--(1,2);
\draw [fill] (-3.5,2) circle [radius=0.1];
\draw [fill] (-6,2) circle [radius=0.1];
\draw [fill] (-0.5,1) circle [radius=0.1];
\draw [fill] (1,1) circle [radius=0.1];
\draw [fill] (2.5,1) circle [radius=0.1];
\draw [fill] (1,0) circle [radius=0.1];
\draw [fill] (-0.5,2) circle [radius=0.1];
\draw [fill] (1,2) circle [radius=0.1];
\draw [fill] (2.5,2) circle [radius=0.1];
\draw [fill] (-2,2) circle [radius=0.1];
\draw [fill] (1,3) circle [radius=0.1];
\node [right] at (1,0) {$1$};
\node [right] at (-0.5,1) {$2$};
\node [right] at (-0.5,2) {$6$};
\node [right] at (1,3) {$30p_1p_2\cdots p_{k-1}=x_n$};
\node [right] at (1,1) {$3$};
\node [right] at (1,2) {$10$};
\node [right] at (2.5,1) {$5$};
\node [right] at (2.5,2) {$15$};
\node [left] at (-2,2) {$2p_1$};
\node [left] at (-3.5,2) {$2p_2$};
\node [left] at (-4.2,2) {$\cdots$};
\node [left] at (-6,2) {$2p_{k-1}$};
\end{tikzpicture}}
}
\subfigure
{{\scalefont{0.8}
\begin{tikzpicture}[scale=1]
\draw (1,0)--(-0.5,1)--(-0.5,2)--(0.92,2.92);
\draw (1,0)--(1,1)--(-0.5,2);
\draw (-0.5,1)--(1,2)--(1,2.9);
\draw (1,0)--(2.5,1)--(2.5,2)--(1.08,2.92);
\draw (1,1)--(2.5,2);
\draw (-0.5,1)--(-2,2)--(1,3);
\draw (-0.5,1)--(-3.5,2)--(1,3);
\draw (-0.5,1)--(-6,2)--(1,3);
\draw (2.5,1)--(1,2);
\draw [fill] (-3.5,2) circle [radius=0.1];
\draw [fill] (-6,2) circle [radius=0.1];
\draw [fill] (-0.5,1) circle [radius=0.1];
\draw [fill] (1,1) circle [radius=0.1];
\draw [fill] (2.5,1) circle [radius=0.1];
\draw [fill] (1,0) circle [radius=0.1];
\draw [fill] (-0.5,2) circle [radius=0.1];
\draw [fill] (1,2) circle [radius=0.1];
\draw [fill] (2.5,2) circle [radius=0.1];
\draw [fill] (-2,2) circle [radius=0.1];
\draw [fill] (1,3) circle [radius=0.1];
\node [right] at (1,0) {$-1$};
\node [right] at (-0.5,1) {$k$};
\node [right] at (-0.5,2) {$-1$};
\node [right] at (1,3) {$1$};
\node [right] at (1,1) {$1$};
\node [right] at (1,2) {$-1$};
\node [right] at (2.5,1) {$1$};
\node [right] at (2.5,2) {$-1$};
\node [left] at (-2,2) {$-1$};
\node [left] at (-3.5,2) {$-1$};
\node [left] at (-4.2,2) {$\cdots$};
\node [left] at (-6,2) {$-1$};
\end{tikzpicture}}
}
\caption{Illustration of the semilattice structure used in the proof of Theorem \ref{powerthm} and the Möbius function values $\mu_S(x_i,x_n)$ that are also needed in the calculation of the function $\Psi_{S,\frac{1}{N^\alpha}}(x_n)$.}\label{fig:laajennettukuutio}
\end{figure}
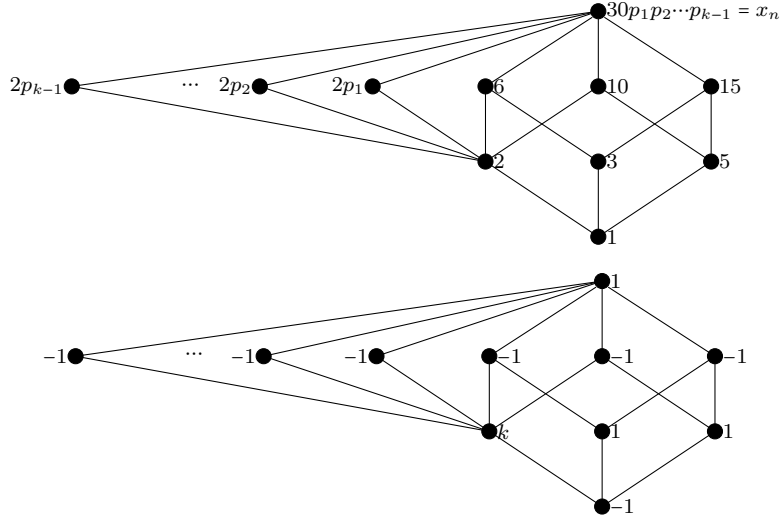

Now we write
\begin{align*}
h(\alpha)&:=\Psi_{S,\frac{1}{N^\alpha}}(x_n)=\sum_{i=1}^n\frac{\mu_S(x_i,x_n)}{x_i^\alpha}\\
&=\frac{-1}{1^\alpha}+\frac{k}{2^\alpha}+\frac{1}{3^\alpha}+\frac{1}{5^\alpha}+\frac{-1}{6^\alpha}+\frac{-1}{10^\alpha}+\frac{-1}{15^\alpha}+\sum_{i=1}^{k-1}\frac{-1}{(2p_i)^\alpha}+\frac{1}{x_n^\alpha}\\
&=\left(\underbrace{-1+\frac{k}{2^\alpha}}_{\geq0\ \text{when}\ \alpha\leq \ln k/\ln 2}\right)+\frac{1}{30^\alpha}(\underbrace{10^\alpha-5^\alpha-3^\alpha-2^\alpha}_{\geq0\ \text{when}\ \alpha\geq 1})\\
&\qquad\qquad+\left(\underbrace{\frac{1}{5^\alpha}-\sum_{i=1}^{k-1}\frac{1}{(2p_i)^\alpha}}_{\geq 0\ \text{for suff. large}\ p_i}\right)+\frac{1}{x_n^\alpha}.
\end{align*}
We can see that $h(\alpha)> 0$ for all $\displaystyle 1\leq \alpha\leq \frac{\ln k}{\ln 2}$. In particular, since $1\leq M\leq \frac{\ln k}{\ln 2}$, we have $h(M)>0$.

On the other hand $h(\alpha)$ is a continuous function of $\alpha$ and we have
\[
\lim_{\alpha\to \infty}h(\alpha)=-1+\lim_{\alpha\to \infty}\sum_{i=2}^n\underbrace{\frac{\mu_S(x_i,x_n)}{x_i^\alpha}}_{\to 0}=-1.
\]
It now follows from Bolzano's theorem that $h(\alpha_0)=0$ for some $\alpha_0>M$.
\end{proof}

\end{document}